\documentclass{amsart}
\usepackage{pifont}
\usepackage{bbm}
\usepackage{bbding}
\usepackage{latexsym}
\usepackage{amsmath}   
\usepackage{tipa}
\usepackage{mathrsfs}  
\usepackage{amssymb}   
\usepackage{color}     
\usepackage{caption2}  
\usepackage{amsfonts}
\usepackage{indentfirst}
\usepackage{amsbsy}
\usepackage{amsthm}
\usepackage{bm}
\usepackage{graphicx}
\usepackage{subfigure}
\usepackage{hyperref}
\usepackage{geometry}
\def\beeq{\begin{equation}}      \def\eneq{\end{equation}}
\def\beeqy{\begin{eqnarray}}     \def\eneqy{\end{eqnarray}}
\def\bece{\begin{center}}        \def\ence{\end{center}}
\def\ba {\begin{array}}          \def\ea {\end{array}}
\def\bess{\begin{eqnarray*}}     \def\eess{\end{eqnarray*}}
\def\bes{\begin{split}}          \def\ens{\end{split}}
\def\bali{\begin{align}}         \def\enali{\end{align}}
\def\bali{\begin{remark}}         \def\enali{\end{remark}}

\def\s1{\sqrt{-1}}

\theoremstyle{plain}

\newtheorem{theorem}{\noindent{\bf Theorem}}[section]
\newtheorem{proposition}[]{\noindent{\bf Proposition}}[section]

\newtheorem{lemma}[]{\noindent{\bf Lemma}}[section]
\newtheorem{corollary}[]{\noindent{\bf Corollary}}[section]
\newtheorem{example}[]{\noindent{\bf Example}}[section]
\newtheorem{remark}[]{\noindent{\bf Remark}}[section]
\newtheorem{conjecture}[]{\noindent{\bf Conjecture}}[section]

        \def\bin#1#2 {{#1\choose#2}}

           \def\dfrac#1#2 {{\displaystyle{#1\over#2}}}
           \def\din#1#2 {{\displaystyle{#1\choose#2}}}

\def\a.{{\rm \ddot{a}}}

\begin{document}
\renewcommand{\thesection}{\arabic{section}}
\renewcommand{\theequation}{\thesection.\arabic{equation}}

\baselineskip=18pt

\title{$\ \ $ A note on rational maps with three branching points on the Riemann sphere}
\author { Zhiqiang Wei,~~Yingyi Wu, ~~Bin Xu}
\maketitle

\begin{abstract}
Studying the existence of rational maps with given branching datum is a classical problem in the field of complex analysis and algebraic geometry. This problem dates back to Hurwitz and remains open to this day. In this paper, we utilize complex analysis to establish a property of rational maps with 3 branching points on the Riemann sphere. Given two compact Riemann surfaces $M$ and $N$, a pair $(d,\mathcal{D})$ of an integer $d\geq2$  and a collection $\mathcal{D}$ of nontrivial partitions of $d$ is called a candidate branching datum if it satisfies the Riemann-Hurwitz formula. And a candidate branching datum is considered exceptional when there is no rational map realization for it. As applications, we present some new types of exceptional branching datum. These results cover some previous  results mentioned in \cite{EKS84,PP06,Zhu19}. We also deduce the realizability of a certain type of candidate branching datum on the Riemann sphere.
  \vspace*{2mm}


\vspace*{2mm}
\noindent{\bf 2020 MR Subject Classification:}\hskip3mm 57M12

\end{abstract}

\section{Background and main results}
Studying the existence of branched covers between two compact Riemann surfaces with given branching datum is an important problem in the field of complex analysis and algebraic geometry. This problem is commonly referred to as the Hurwitz existence problem. In other words, given two compact Riemann surfaces $M$ and $N$, along with a collection $\mathcal{D}$ of partitions of a positive integer $d$, the question is whether there exists a degree-$d$ branched cover $f:M\rightarrow N$ with $\mathcal{D}$ as branching datum.\par
Let $M$ and $N$ be compact Riemann surfaces, and let their respective Euler characteristics be denoted by $\chi(M)$ and $\chi(N)$. A smooth map $f:M\rightarrow N$ is a degree-$d$ branched cover if for each $x\in N$ there is a partition $\mu(x)=[\alpha_{1},\ldots,\alpha_{r}]$ of $d$ such that, over a neighborhood of $x$ in $N$, $f$ is equivalent to the map $\widetilde{f}:\{1,\ldots,r\}\times \mathbb{D}\rightarrow \mathbb{D}$ where $\widetilde{f}(j,z)=z^{\alpha_{j}}$ and $x$ corresponds to $0$ in the disc $\mathbb{D}=\{z| |z|<1\}\subseteq\mathbb{C}$ (here the square brackets are used to denote an unordered set with repetitions). The points $x\in N$ for which $\mu(x)$ is not the trivial partition $[1,1,\ldots,1]$ of $d$ constitute the branching set $B_{f}$ of $f$, and it is a finite set. The collection $\mathcal{D}=\{\mu(x)|x\in B_{f}\}$ (with repetitions allowed) is called the branching datum of $f$. As is well known, the degree $d$, Euler characteristics $\chi(M),\chi(N)$ and the branching datum $\mathcal{D}$ of $f$ should satisfy the Riemann-Hurwitz formula
\begin{equation}\label{RHF}
\nu(\mathcal{D})=d\cdot\chi(N)-\chi(M),
\end{equation}
where $\nu(\mathcal{D})$ denotes the total branching of $f$, which is defined as follows. Let $B_{f}=\{x_{1},\ldots,x_{n}\}\subseteq N$ and $\mathcal{D}=\{[\alpha_{1}^{1},\ldots,\alpha_{1}^{r_{1}}],\ldots, [\alpha_{n}^{1},\ldots,\alpha_{n}^{r_{n}}]\}$ represent the branching set and the branching datum of $f$, respectively. For each $x_{k}$, its pre-image under $f$ consists of a finite number of points $y_{k}^{1},\ldots,y_{k}^{r_{k}}\in M$, and near each $y_{k}^{j}$ the map $f$ is equivalent to  $\widetilde{f}(z)=z^{\alpha_{k}^{j}}$.  The integer $\alpha_{k}^{j}$ is usually referred to as the local degree or multiplicity of $f$ at the point $y_{k}^{j}$. Since $x_{k}$ is a branching point, at least one of the $\alpha_{k}^{j}$'s should be greater than $1$.  The total branching of $f$ is then defined as $\nu(\mathcal{D})=\sum\limits_{k=1}^{n}\sum\limits_{j=1}^{r_{k}}(\alpha_{k}^{j}-1)$. Since
\begin{equation}\label{DC}
\sum_{j=1}^{r_{k}}\alpha_{k}^{j}=d,~\forall k=1,\ldots,n,
\end{equation}
the Riemann-Hurwitz formula (\ref{RHF}) can be expressed as
\begin{equation}\label{RHF-1}
\sum_{k=1}^{n}(d-r_{k})=d\cdot\chi(N)-\chi(M).
\end{equation}

Given two compact Riemann surfaces $M$ and $N$, a pair $(d,\mathcal{D})$ of an integer $d\geq2$  and a collection $\mathcal{D}$ of nontrivial partitions of $d$ is called a candidate branching datum if it satisfies the Riemann-Hurwitz formula. The Hurwitz problem asks that given two compact Riemann surfaces $M$ and $N$ and a candidate branching datum $(d,\mathcal{D})$ whether there exists a degree-$d$ branched cover $f:M\rightarrow N$ with $(d,\mathcal{D})$ as branching datum. In his classical work \cite{Hur91}, Hurwitz reduced the problem to a problem involving partitions realized by suitable permutations in symmetric groups. In \cite{EKS84}, Edmonds, Kulkarni, and Stong proved that all candidate branching datum is realizable when $\chi(N)\leq0$. However, when $N=S^{2}$ the problem becomes much more complex. It is well known that there exist candidate branching datum $(d,\mathcal{D})$ that cannot be realized by a branched cover, hence called exceptional. For example, this is the case for $d=4$ and $\mathcal{D}=\{[3,1],[2,2],[2,2]\}$. Characterizing all of such exceptional candidate branching datum remains an open problem to this day.  In \cite{Zhe06}, Zheng determined all exceptional candidate branching datum with $n=3$ and $d\leq 22$ by computer for the cases where $M=N=S^{2}$ and $M=T^{2},N=S^{2}$. \par

Finding new types of exceptional candidate branching datum is of interest as it may provide insights towards establishing a universal criterion. However, the general pattern of realizable data remains unclear. Various approaches such as dessins d'enfant, Speiser graphs and the monodromy approach have been explored to attack the problem. For more details, see \cite{Bar01, BP23, CH22, EKS84, Ger87, Hus62, KZ96, Med84, Med90, MSS04, OP06, Pak09, P20, PP09, PP12, PP06, PP07, PP08, SX20, Zhe06} and the references cited therein for more results. Particularly, we refer the reader to \cite{CH22, P20, PP08} for a review of available results and techniques. In \cite{EKS84}, a conjecture proposing connections with number-theoretic facts has been put forward, which is supported by strong evidence in \cite{Pak09, PP09}.

\begin{conjecture}[Prime degree conjecture]
 Suppose $(d,\mathcal{D})$ is a candidate branching datum. If $d$ is a prime, then $(d,\mathcal{D})$ is realizable.
\end{conjecture}

In \cite{EKS84}, Edmonds, Kulkarni, and Stong reduced the Prime degree conjecture to the case of candidate branching datum with exactly 3 partitions. Up to now, this conjecture is open. In this paper we investigate the properties of branched covers for the cases where $M=N=S^{2}$. Regarding $S^{2}\cong\overline{\mathbb{C}}$ as the Riemann sphere, by Sto\"{i}low's theorem (see \cite{S28} or \cite{RP19}), the Hurwitz problem reduces to the existence of rational maps on $\overline{\mathbb{C}}$ with the given branching datum. Thus we can use complex analysis techniques to attack the problem. Motivation by the recent work of Song and Xu \cite{SX20}, we use their idea of expressing a rational map as a composition of a rational map and a power map. Our main result is as follows and can be viewed as a generalization of a result in \cite{PP06}(\textbf{Theorem} 1.4 and \textbf{Theorem} 1.5).

\begin{theorem}\label{Main-th1}
  Suppose $f:\overline{\mathbb{C}}\rightarrow \overline{\mathbb{C}}$ is a rational map of degree  $d=rk, r\geq 2, k\geq2$, with 3 branching points with associated branching datum
  $$\{\mu_{1}=[\alpha_{1},\ldots,\alpha_{A}],\mu_{2}=[rx_{1},\ldots,rx_{B}],\mu_{3}=[ry_{1},\ldots,ry_{C}]\}.$$
  Then, up to two M\"{o}bius transformations on $\overline{\mathbb{C}}$, there is a rational map  $g:\overline{\mathbb{C}}\rightarrow \overline{\mathbb{C}}$ of degree $k$ such that $f=g^{r}$, with the property that its branching datum consists of $[x_{1},\ldots,x_{B}]$ and $[y_{1},\ldots,y_{C}]$, together with data arising by splitting the partition of $[\alpha_{1},\ldots,\alpha_{A}]$ into $r$ partitions of $k$.
\end{theorem}

\begin{remark}
In \textbf{Theorem} \ref{Main-th1}, the sentence `` up to two M\"{o}bius transformations'' means that $f=\varphi\circ g^{k} \circ\phi$, where $\varphi,\psi:\overline{\mathbb{C}}\rightarrow\overline{\mathbb{C}}$ are two M\"{o}bius transformations.
\end{remark}

\begin{remark}
In \textbf{Theorem} \ref{Main-th1}, if $x_{1}=\ldots=x_{B}=1$ or $y_{1}=\ldots=y_{C}=1$ then the corresponding partition of $k$ should be expunged from the branching datum of $g$.
\end{remark}

Then, under the assumption of \textbf{Theorem} \ref{Main-th1}, one would have immediate corollaries for realizable data.

\begin{corollary}\label{Co-1}
Each $\alpha_{j}\leq k$.
\end{corollary}

\begin{corollary}\label{Co-2}
If $A=2$, then $r=2,\mu_{1}=[k,k]$.
\end{corollary}

\begin{corollary}\label{Co-3}
If $A=r$, then $\mu_{1}=[k,\ldots,k]$.
\end{corollary}
As applications of these results, we offer new proofs of some of the results presented in \cite{PP06,Zhu19} and give some new type of exceptional candidate branching datum (\textbf{Propositions} \ref{P-n-1} and \ref{P-n-2} ).

\begin{proposition}[\cite{Zhu19}]\label{Pro-1}
Suppose $d=2k$ with $k\geq2$, then a candidate branching datum of the form
$$(d,\{[k_{1},k_{2}],[\underbrace{2,\ldots,2}_{k}],[\underbrace{2,\ldots,2}_{k}]\})$$
is exceptional if $k_{1}\neq k_{2}$. Suppose in addition $k\geq3$, then a candidate branching datum of the form
$$(d,\{[\underbrace{2,\ldots,2}_{k}],[\underbrace{2,\ldots,2}_{j_{1}},2k-2j_{1}],[\underbrace{2,\ldots,2}_{j_{2}},2k-2j_{2}]\})$$
is exceptional if $j_{1}\neq j_{2}$.
\end{proposition}
\begin{proof}
Since $k_{1}\neq k_{2}$, so by \textbf{Corollary} \ref{Co-2}, $(d,\mathcal{D})$ is exceptional.
If $k\geq3$ , since $j_{1}+j_{2}=k$ and $j_{1}\neq j_{2}$, then $2j_{1}<k$ or $2j_{2}<k$. If $2j_{1}<k$, then $2k-2j_{1}>k$; If $2j_{2}<k$, then $2k-2j_{2}>k$, so by \textbf{Corollary} \ref{Co-1}, $(d,\mathcal{D})$ is exceptional.

\end{proof}

\begin{proposition}[\cite{Zhu19}]
Suppose $d=3k$ with $k\geq 3$,
then a candidate branching datum of the form
$$(d,\{[k-2,\underbrace{2,\ldots,2}_{k+1}],[\underbrace{3,\ldots,3}_{k}],[\underbrace{3,\ldots,3}_{k}]\})$$
is exceptional if $k$ is odd.
\end{proposition}
\begin{proof}
Since $k$ is odd and $2$ is even, there is just one odd in $[k-2,2,\ldots,2]$, so $[k-2,2,\ldots,2]$ cannot be split into 3 partitions of $k$. Therefor, by \textbf{Theorem} \ref{Main-th1}, $(d,\mathcal{D})$ is exceptional.
\end{proof}

\begin{proposition}[\cite{Zhu19}]
Suppose $d=rk$ with $r\geq2,k\geq2$, then a candidate branching datum of the form
 $$(d, \{[j_{1},j_{2},\underbrace{1,\ldots,1}_{(r-2)k}],[\underbrace{r,\ldots,r}_{k}],[\underbrace{r,\ldots,r}_{k}]\})$$
is exceptional if $j_{1}\neq j_{2}$.
\end{proposition}
\begin{proof}
Since $j_{1}+j_{2}=2k$  and $j_{1}\neq j_{2}$, then $j_{1}>k$ or $j_{2}>k$, so by \textbf{Corollary} \ref{Co-1}, $(d,\mathcal{D})$ is exceptional.
\end{proof}

\begin{proposition}\label{P-n-1}
Suppose $d=3k$ with $k\geq 3$ and $k\equiv 1~mod~2$ , then a candidate branching datum of the form
$$(d,\{[j_{1},j_{2},\underbrace{2,\ldots,2}_{k}],[\underbrace{3,\ldots,3}_{k}],[\underbrace{3,\ldots,3}_{k}]\}),$$
is exceptional if $j_{1},j_{2}$ have different parity.
\end{proposition}
\begin{proof}
Since $j_{1},j_{2}$ have different parity, then there is just one odd in the partition $[j_{1},j_{2},\underbrace{2,\ldots,2}_{k}]$. Since $k$ is odd, then $[k-2,2,\ldots,2]$ cannot be split into 3 partitions of $k$, so by \textbf{Theorem} \ref{Main-th1}, $(d,\mathcal{D})$ is exceptional.
\end{proof}

\begin{proposition}\label{P-n-2}
Suppose $d=3k$ with $k=2+3l, l\geq1$,  then a candidate branching datum of the form
$$(d,\{[\underbrace{3,\ldots,3}_{k-1},\underbrace{1,1,1}_{3}],[\underbrace{3,\ldots,3}_{k}],[\underbrace{3,\ldots,3}_{k}]\}$$
is exceptional.
\end{proposition}
\begin{proof}
Since $k=2+3l\geq 5$ and the number of 1's in $[\underbrace{3,\ldots,3}_{k-1},\underbrace{1,1,1}_{3}]$ equals 3, then $[\underbrace{3,\ldots,3}_{k-1},\underbrace{1,1,1}_{3}]$ cannot be split into 3 partitions of $k$, so by \textbf{Theorem} \ref{Main-th1}, $(d,\mathcal{D})$ is exceptional.
\end{proof}

From the constructed exceptional data mentioned above, we obtain an additional result regarding non-prime degrees on the case $M=N=S^{2}$, which has been proven in \cite{EKS84,Zhu19} using a different methodology.

\begin{corollary}[\cite{EKS84},\cite{Zhu19}]
For every non-prime $d$, there exists at least one candidate branching data $(d,\mathcal{D})$ that is exceptional.
\end{corollary}

\section{Proof of \textbf{Theorem} \ref{Main-th1}}
Firstly, we prove the following lemma.
\begin{lemma}\label{Le-1}
Suppose $g:\mathbb{C} \rightarrow \mathbb{C}$ is a non-constant holomorphic function. For a point  $z_{0}\in \mathbb{C}$, if $g(z_{0})\neq0$, then $g$ and $g^{r}$, for every integer $r$, have the same local degree at $z_{0}$.
\end{lemma}
\begin{proof}
Call $n$ the local degree of $g$ at $z_{0}$. Then, knowing that $g(z_{0})\neq0$, the order-$n$ Taylor expansion of $g$ at $z_{0}$ has the form
$$g(z)=a+b (z-z_{0})^{n}+\textrm{o}((z-z_{0})^{n})$$
with $a,b\neq0$. It follows that
$$g^{r}(z)=a^{r}+rab (z-z_{0})^{n}+\textrm{o}((z-z_{0})^{n}),$$
and $rab\neq0$, so $g^{r}$ also has local degree $n$ at $z_{0}$.
\end{proof}

Now the proof of \textbf{Theorem} \ref{Main-th1} is as follows.\par
Since any 3 points in $\overline{\mathbb{C}}$ can be transfer to any 3 points in $\overline{\mathbb{C}}$ by using a M\"{o}bius transformation, then we can use a M\"{o}bius transformation acting on the target $\overline{\mathbb{C}}$ such that the 3 branching points of $f$ are $0,-1,\infty$, and use a M\"{o}bius transformation acting the source $\overline{\mathbb{C}}$ such that $\infty\overline{\in}f^{-1}(0)\cup f^{-1}(-1)\cup f^{-1}(\infty)$. Thus we can suppose the expression of $f$ is
$$f(z)=\frac{(z-z_{1})^{rx_{1}}(z-z_{2})^{rx_{2}}\cdots (z-z_{B})^{rx_{B}}}{(z-w_{1})^{ry_{1}}(z-w_{2})^{ry_{2}}\cdots (z-w_{C})^{ry_{C}}},$$
where $z_{1},\ldots,z_{B},w_{1},\ldots,w_{C}\in\mathbb{C}$ are distinct complex numbers.\par
Set $g(z)=\frac{(z-z_{1})^{x_{1}}(z-z_{2})^{x_{2}}\cdots (z-z_{B})^{x_{B}}}{(z-w_{1})^{y_{1}}(z-w_{2})^{y_{2}}\cdots (z-w_{C})^{y_{C}}}$. Then $g$ is a rational map with degree-$k$ on $\overline{\mathbb{C}}$. It is obvious that $f=g^{r}$, $0$ is a branching point of $g$ if and only if $x_{j}\geq2$ for some $j$ and $\infty$ is a branching point of $g$ if and only if $y_{l}\geq2$ for some $l$. \par

Suppose $z_{0}\in\mathbb{C}\setminus\{z_{1},\ldots,z_{B},w_{1},\ldots,w_{C}\}$, then $g(z_{0})\neq0$, so by \textbf{Lemma} \ref{Le-1}, $g$ and $f=g^{r}$ have the same local degree at $z_{0}$. Knowing that the degree of $g$ is k, thus the branching datum of $g$ is
$$\{[\alpha_{1}^{1},\ldots,\alpha_{1}^{l_{1}}],\ldots,[\alpha_{s}^{1},\ldots,\alpha_{s}^{l_{s}}],[x_{1},\ldots,x_{B}], [y_{1},\ldots,y_{C}]\},$$
for some $1\leq s\leq r$, where $\alpha_{1}^{1},\ldots,\alpha_{1}^{l_{1}},\ldots,\alpha_{s}^{1},\ldots,\alpha_{s}^{l_{s}}$ belong to $\alpha_{1},\ldots,\alpha_{A}$ which means that up to a permutation $(\alpha_{1}^{1},\ldots,\alpha_{1}^{l_{1}},\ldots,\alpha_{s}^{1},\ldots,\alpha_{s}^{l_{s}},1,\ldots,1)=(\alpha_{1},\ldots,\alpha_{A})$.

We note that by \textbf{Lemma} \ref{Le-1}, the proof of the following theorem is easy.
\begin{theorem}\label{Main-th5}
Suppose $g:\overline{\mathbb{C}}\rightarrow\overline{\mathbb{C}}$ is a rational map with degree $d\geq 3$ and branching datum
$$\{[\alpha_{1},\ldots,\alpha_{A}],[\beta_{1},\ldots,\beta_{B}],[\gamma_{1},\ldots,\gamma_{C}]\},$$
then, for any integer $k\geq2$, there is a M\"{o}bius transformation $\varphi$ such that $f=(\varphi\circ g)^{k}$ is a rational map with degree $kd$ and branching datum
  $$\{[\alpha_{1},\ldots,\alpha_{A},\underbrace{1,\ldots,1}_{(k-1)d}],[k\beta_{1},\ldots,k\beta_{B}],[k\gamma_{1},\ldots,k\gamma_{C}]\}.$$
\end{theorem}

\section{Some existence results}
 We conclude the paper by presenting some existence results for a certain type of rational maps on $\overline{\mathbb{C}}$ with 3 branching points. First, we give a new proof of the following theorem which was proved in \cite{EKS84} and which is a special case in \cite{CH22}(\textbf{Theorem} 3.5).

\begin{theorem}[\cite{EKS84}]\label{Main-th2}
  There exists a rational map $f:\overline{\mathbb{C}}\rightarrow\overline{\mathbb{C}}$ with branching datum
  $$\{[\alpha_{1},\alpha_{2}],[\underbrace{2,\ldots,2}_{k}],[\underbrace{2,\ldots,2}_{k}]\},$$
  where $k\geq2,\alpha_{1}+\alpha_{2}=2k$, if and only if $\alpha_{1}=\alpha_{2}=k$.
\end{theorem}
\begin{proof}
The necessary is from \textbf{Corollary} \ref{Co-2}.\par
 ``If ". Set $g(z)=\frac{z^{k}-1}{z^{k}+1}$. By direct calculation, the branching points of $g$ are $-1,1$, $g^{-1}(-1)=\{0\},g^{-1}(1)=\{\infty\}$ and the branching datum of $g$ is $\{[k],[k]\} $. Set $f=g^{2}$. Then the branching points of $f$ are $1,0,\infty$ and the branching datum of $f$ is
 $$\{[k,k],[\underbrace{2,\ldots,2}_{k}],[\underbrace{2,\ldots,2}_{k}]\}.$$

\end{proof}

Secondly, we can derive the following theorem.
\begin{theorem}\label{Main-th3}
  Suppose that $k\geq 3$. There exists a degree-$2k$ rational map $f:\overline{\mathbb{C}}\rightarrow\overline{\mathbb{C}}$ with branching datum
  $$\{[\alpha_{1},\alpha_{2},\alpha_{3}],[\underbrace{2,\ldots,2}_{k}],[\underbrace{2,\ldots,2}_{k-2},4]\},$$
 if and only if one of the $\alpha_{j}$'s is $k$.
\end{theorem}
\begin{proof}
``If ". Without loss of generality, we can suppose that $\alpha_{1}=k$. By a result of Boccara \cite{Bo82} or Thom \cite{Th65}, there exists a rational map $g:\overline{\mathbb{C}}\rightarrow\overline{\mathbb{C}}$ with branching datum
$$\{[k],[\alpha_{2},\alpha_{3}],[\underbrace{1,\ldots,1}_{k-2},2]\}.$$
Up to a M\"{o}bius transformation, we may suppose the branching points of $g$ are $-1,1,0$ corresponding the partitions $[k],[\alpha_{2},\alpha_{3}],[\underbrace{1,\ldots,1}_{k-2},2]$, respectively.
Then $f=g^{2}$ is the desired function.\par
``Only if ". If $f:\overline{\mathbb{C}}\rightarrow\overline{\mathbb{C}}$ is a rational map with branching datum
 $$\{[\alpha_{1},\alpha_{2},\alpha_{3}],[\underbrace{2,\ldots,2}_{k}],[\underbrace{2,\ldots,2}_{k-2},4]\},$$
then, by \textbf{Theorem} \ref{Main-th1}, $[\alpha_{1},\alpha_{2},\alpha_{3}]$ can be split into 2 partitions of $k$. Thus one of the $\alpha_{j}$'s is $k$.
\end{proof}

In general, we obtain the following theorem.
\begin{theorem}
 Suppose that $k\geq3$ and $x\geq1$. There exists a degree-$2k$ rational map $f:\overline{\mathbb{C}}\rightarrow\overline{\mathbb{C}}$ with branching datum
  $$\{[\alpha_{1},\alpha_{2},\ldots,\alpha_{x+1}],[\underbrace{2,\ldots,2}_{k}],[\underbrace{2,\ldots,2}_{k-x},2x]\},$$
  if and only if $[\alpha_{1},\alpha_{2},\ldots,\alpha_{x+1}]$ can be split into two partitions of $k$ and $\frac{k}{{\rm GCD}(\alpha_{1},\alpha_{2},\ldots,\alpha_{x+1})}\geq x$.
\end{theorem}
\begin{proof}
``Only if ". If $f:\overline{\mathbb{C}}\rightarrow\overline{\mathbb{C}}$ is a rational map with branching datum
  $$\{[\alpha_{1},\alpha_{2},\ldots,\alpha_{x+1}],[\underbrace{2,\ldots,2}_{k}],[\underbrace{2,\ldots,2}_{k-x},2x]\},$$
then, by \textbf{Theorem} \ref{Main-th1}, there exists a rational map $g$ with branching datum
$$\{[\alpha_{1}^{1},\ldots,\alpha_{1}^{r_{1}}],[\alpha_{2}^{1},\ldots,\alpha_{2}^{r_{2}}],[1,\ldots,1,x]\}$$
such that $f=g^{2}$, where$ [\alpha_{1}^{1},\ldots,\alpha_{1}^{r_{1}},\alpha_{2}^{1},\ldots,\alpha_{2}^{r_{2}}]=[\alpha_{1},\alpha_{2},\ldots,\alpha_{x+1}]$. By a result of Boccara \cite{Bo82}, $\frac{k}{{\rm GCD}(\alpha_{1},\alpha_{2},\ldots,\alpha_{x+1})}\geq x$.\par
``If ". Suppose $[\alpha_{1},\alpha_{2},\ldots,\alpha_{x+1}]$ can be split into two partitions of $k$. Denote the two partitions by $[\alpha_{1}^{1},\ldots,\alpha_{1}^{r_{1}}]$ and $[\alpha_{2}^{1},\ldots,\alpha_{2}^{r_{2}}]$. Since
$${\rm GCD}(\alpha_{1}^{1},\ldots,\alpha_{1}^{r_{1}},\alpha_{2}^{1},\ldots,\alpha_{2}^{r_{2}})={\rm GCD}(\alpha_{1},\alpha_{2},\ldots,\alpha_{x+1}),$$
then
$$\frac{k}{{\rm GCD}(\alpha_{1}^{1},\ldots,\alpha_{1}^{r_{1}},\alpha_{2}^{1},\ldots,\alpha_{2}^{r_{2}})}\geq x.$$
So by a result of Boccara \cite{Bo82}, there exists a rational map $g:\overline{\mathbb{C}}\rightarrow\overline{\mathbb{C}}$ with branching datum
$$\{[\alpha_{1}^{1},\ldots,\alpha_{1}^{r_{1}}],[\alpha_{2}^{1},\ldots,\alpha_{2}^{r_{2}}],[1,\ldots,1,x]\}.$$
Up to a M\"{o}bius transformation, we may suppose the branching points of $g$ are $-1,1,0$ corresponding the partitions $[\alpha_{1}^{1},\ldots,\alpha_{1}^{r_{1}}],[\alpha_{2}^{1},\ldots,\alpha_{2}^{r_{2}}],[1,\ldots,1,x]$, respectively. Then $f=g^{2}$ is the desired function.\par

\end{proof}

Finally, we provide more explicit examples to explain our results.
\begin{example}
There exists a rational map $f:\overline{\mathbb{C}}\rightarrow\overline{\mathbb{C}}$ with branched data
  $$\{[\alpha_{1},\alpha_{2},\alpha_{3},2,2],[3,3,3],[3,3,3]\},$$
  where $\alpha_{1}+\alpha_{2}+\alpha_{3}=5,\alpha_{1}\geq\alpha_{2}\geq\alpha_{3}$, if and only if $\alpha_{1}=3$.
\end{example}
\begin{proof}
``Only if " . Obviously $\alpha_{1}=2$ or $3$. If $\alpha_{1}=2$, then $\alpha_{2}=2,\alpha_{3}=1$. So $[\alpha_{1},\alpha_{2},\alpha_{3},2,2]=[2,2,1,2,2]$. But $[2,2,1,2,2]$ cannot be split into 3 partitions of $3$. Thus $\alpha_{1}=3$.\par
``If ". If $\alpha_{1}=3$, then $\alpha_{2}=\alpha_{3}=1$.  By a result of Boccara \cite{Bo82} or Thom \cite{Th65} there exists a rational map $g$ with degree $3$ and branching datum
$$\{[3],[2,1],[2,1]\}.$$
Using two M\"{o}bius transformations, we can suppose the branching points of $g$ are $1,e^{\frac{2\pi}{3}i},e^{\frac{4\pi}{3}i}$ and $\infty\overline{\in}f^{-1}(1)\cup f^{-1}(e^{\frac{2\pi}{3}i})\cup f^{-1}(e^{\frac{4\pi}{3}i})$, then $f=g^{3}$ is a rational map with degree $9$, branching points $1,0,\infty$ and branching datum
 $$\{[3,1,1,2,2],[3,3,3],[3,3,3]\}.$$
\end{proof}

\begin{example}
There exists a rational map $f:\overline{\mathbb{C}}\rightarrow\overline{\mathbb{C}}$ with branching datum
  $$\{[\alpha_{1},\alpha_{2},\alpha_{3},2,2,2,2],[3,3,3,3,3],[3,3,3,3,3]\},$$
  where $\alpha_{1}+\alpha_{2}+\alpha_{3}=7,\alpha_{1}\geq\alpha_{2}\geq\alpha_{3}$, if and only if $\alpha_{1}=5$ or $\alpha_{1}=\alpha_{2}=3$.
\end{example}
\begin{proof}
``Only if " . Obviously $\alpha_{1}=3,4$ or $5$. When $\alpha_{1}=3$, then $\alpha_{2}=2,\alpha_{3}=2$ or $\alpha_{2}=3,\alpha_{3}=1$. When $\alpha_{1}=3,\alpha_{2}=2$ and $\alpha_{3}=2$, since $[3,2,2,2,2,2,2]$ cannot be split into 3 partitions of $5$, so we exclude this case. When $\alpha_{1}=3,\alpha_{2}=3$ and $\alpha_{3}=1$, $[3,3,1,2,2,2,2]$ can be split into 3 partitions of $5$. When $\alpha_{1}=4$, since $\alpha_{2}+\alpha_{3}=3$, then there is just one odd in $[4,\alpha_{2},\alpha_{3},2,2,2,2]$, so $[4,\alpha_{2},\alpha_{3},2,2,2,2]$ cannot be split into 3 partitions of $5$. When $\alpha_{1}=5$, then $\alpha_{2}=\alpha_{3}=1$, and $[5,1,1,2,2,2,2]$ can be split into 3 partitions of $5$. In summary, $\alpha_{1}=5$ or $\alpha_{1}=\alpha_{2}=3$.\par
``If ". If $\alpha_{1}=5$, then $\alpha_{2}=\alpha_{3}=1$.  By a result of Boccara \cite{Bo82} or Thom \cite{Th65} there exists a rational map $g$ with degree $5$ and branching datum
$$\{[5],[2,2,1],[2,2,1]\}.$$
Using two M\"{o}bius transformations, we can suppose the branching points of $g$ are $1,e^{\frac{2\pi}{3}i},e^{\frac{4\pi}{3}i}$ and $\infty\overline{\in}g^{-1}(1)\cup g^{-1}(e^{\frac{2\pi}{3}i})\cup g^{-1}(e^{\frac{4\pi}{3}i})$, then $f=g^{3}$ is a rational map with degree $15$, branching points $1,0,\infty$ and branching datum
 $$\{[5,1,1,2,2,2,2],[3,3,3,3,3],[3,3,3,3,3]\}.$$

If $\alpha_{1}=\alpha_{2}=3$, then $\alpha_{3}=1$.  By a result of Zheng \cite{Zhe06} there exists a rational map $g$ with degree $5$ and branching datum
$$\{[3,2],[3,2],[2,2,1]\}.$$
Using two M\"{o}bius transformations, we can suppose the branching points of $g$ are $1,e^{\frac{2\pi}{3}i},e^{\frac{4\pi}{3}i}$ and $\infty\overline{\in}f^{-1}(1)\cup f^{-1}(e^{\frac{2\pi}{3}i})\cup f^{-1}(e^{\frac{4\pi}{3}i})$, then $f=g^{3}$ is a rational map with degree $15$, branching points $1,0,\infty$ and branching datum
 $$\{[3,3,1,2,2,2,2],[3,3,3,3,3],[3,3,3,3,3]\},$$
\end{proof}

\section*{Acknowledgments}
~~~~
We would like to express our gratitude to the reviewers for their constructive comments and suggestions, which have helped to improve the quality and clarity of this manuscript. This work was initiated during Wei's visit to the Chern Institute of Mathematics in the fall of 2023. Wei would like to express his gratitude to Professor Zizhou Tang for the warm hospitality he received. Wei is partially supported by the National Natural Science Foundation of China (Grant No. 12171140). Wu is partially supported by the Fundamental Research Funds for the Central Universities and the CAS Project for Young Scientists in Basic Research (Grant No. YSBR-001). Xu is supported in part by the National Natural Science Foundation of China (Grant Nos. 12271495, 11971450, and 12071449) and the CAS Project for Young Scientists in Basic Research (Grant No. YSBR-001).

\textbf{Declarations}

\textbf{Data Availability Statement}  This manuscript has no associated data.

\textbf{Competing interests} On behalf of all authors, the corresponding author declares that there is no conflict of interest.


\noindent
Zhiqiang Wei\\
School of Mathematics and Statistics, Henan University, Kaifeng 475004 P.R. China\\
Center for Applied Mathematics of Henan Province, Henan University, Zhengzhou 450046 P.R. China\\
Email: weizhiqiang15@mails.ucas.edu.cn. ~or~10100123@vip.henu.edu.cn.\\
Yingyi Wu\\
School of Mathematical Sciences, University of Chinese Academy of Sciences, Beijing 100049, P.R. China\\
Email: wuyy@ucas.ac.cn\\
Bin Xu\\
CAS Wu Wen-Tsun Key Laboratory of Mathematics and School of Mathematical Sciences\\
University of Science and Technology of China, Hefei 230026, P.R. China\\
Email: bxu@ustc.edu.cn.

\end{document}